\documentclass[a4paper,12pt,reqno]{article}
\usepackage{amssymb, hyperref}
\usepackage{amsthm}
\usepackage{amsmath}
\usepackage{latexsym}
\usepackage{epsfig}
\usepackage{amscd}
\usepackage{graphicx,color}
\newtheorem{theorem}{Theorem}
\newtheorem{lemma}[theorem]{Lemma}
\newtheorem{cor}[theorem]{Corollary}

\theoremstyle{definition}
\newtheorem{definition}[theorem]{Definition}
\newtheorem{example}[theorem]{Example}

\theoremstyle{remark}

\newcommand{\legendre}[2]{\genfrac{(}{)}{}{}{#1}{#2}}

\newcommand{\blankbox}[2]{%
\parbox{\columnwidth}{\centering
}%
}

\title{Complete Generalized Fibonacci Sequences Modulo Primes}
 
\author{Mohammad Javaheri, Nikolai A. Krylov\\ ~ \\
Siena College, Department of Mathematics\\
515 Loudon Road, Loudonville NY 12211, USA\\ ~ \\
mjavaheri@siena.edu, nkrylov@siena.edu
}
\date{}

\begin{document}
\maketitle
\begin{abstract}
We study generalized Fibonacci sequences $F_{n+1}=PF_n-QF_{n-1}$ with initial values $F_0=0$ and $F_1=1$. Let $P,Q$ be nonzero integers such that $P^2-4Q$ is not a perfect square. We show that if $Q=\pm 1$ then the sequence $\{F_n\}_{n=0}^\infty$ misses a congruence class modulo every prime large enough. On the other hand, if $Q \neq \pm 1$, we prove that (under GRH) the sequence $\{F_n\}_{n=0}^\infty$ hits every congruence class modulo infinitely many primes. 
\end{abstract}
A generalized Fibonacci sequence with integer parameters $(P,Q)$ is a second order homogeneous 
difference equation generated by 
\begin{equation}\label{gfequation}
F_{n+1} = P F_{n} -Q F_{n-1},
\end{equation}
$n\geq 1$, with initial values $F_0=0$ and $F_1=1$; we denote this sequence by $F=[P,Q]$. In this paper, we are interested in studying the completeness problem for generalized Fibonacci sequences modulo primes. 

\begin{definition}
A sequence $\{F_n\}_{n=0}^\infty$ is said to be complete modulo a prime $p$, if the map $\mathfrak{e}: \mathbb{N}_0 \rightarrow \mathbb{Z}_p=\mathbb{Z}/p \mathbb{Z}$ defined by $\mathfrak{e}(n)=F_n +p \mathbb{Z}$ is onto. In other words, the sequence is complete modulo $p$ iff
$$\forall k \in \mathbb{Z}~ \exists n\geq 0~F_n \equiv k \pmod p.$$
\end{definition}

The study of completeness problem for the Fibonacci sequence (with parameters $(1,-1)$) began with Shah \cite{Shah} who showed that the Fibonacci sequence is not complete modulo primes $p \equiv 1, 9 \pmod{10}$. Bruckner \cite{Bruckner} proved that the Fibonacci sequence is not complete modulo every prime $p>7$. Somer \cite{Somer} proved that if $p \nmid (P^2+4)$ then $[P,-1]$ is not complete modulo $p$ for every $p>7$ with $p \not\equiv 1, 9 \pmod{20}$. Schinzel \cite{Schinzel} improved Somer's result for all other congruence classes; see also \cite{li}. 

A primitive root modulo a prime $p$ is any integer that generates the multiplicative group modulo $p$. Somer \cite{Somer} observed that, if $Q$ is a primitive root modulo prime $p$ and $P^2-4Q$ is not a quadratic residue modulo $p$, then $[P,Q]$ is complete modulo $p$. Therefore, if $Q$ is a primitive root modulo infinitely many primes, then $[P,Q]$ is complete modulo infinitely many primes. 

Artin's conjecture states that if $Q \neq - 1$ and $Q$ is not a perfect square, then $Q$ is a primitive root modulo infinitely many primes. Hooley \cite{Hooley} proved that Artin's conjecture follows form the Generalized Riemann Hypothesis (GRH). Our main result is Theorem \ref{main} which, under GRH, establishes a dichotomy regarding completeness modulo infinitely many primes. 
\begin{theorem}\label{main}
Let $P,Q$ be nonzero integers and suppose that $P^2-4Q$ is not a perfect square. Then the following statements hold.
\begin{itemize}
\item[(i)] If $Q =\pm 1$ then $[P,Q]$ is complete modulo only finitely many primes. 
\item[(ii)] (GRH) If $Q \neq \pm 1$ then $[P,Q]$ is complete modulo infinitely many primes. 
\end{itemize}
\end{theorem}

In Section \ref{prelemsection}, we briefly discuss the cases not covered by Theorem \ref{main} i.e., when $PQ=0$ or $P^2-4Q$ is a perfect square. As we stated before, Schinzel \cite{Schinzel} proved part (i) of Theorem \ref{main} for the case $Q=-1$. We prove part (i) for the case $Q=1$ in Section \ref{main}, where we show that $[P,1]$ is complete modulo a prime $p>3$ if and only if $P \equiv \pm 2 \pmod p$ (Theorem \ref{caseq1}). 

Somer \cite{Somer} proved part (ii) of Theorem \ref{main} in the case $Q$ is not a perfect square. In Section \ref{main}, we prove part (ii) in the remaining case where $Q$ is a perfect square not equal to 1 (Lemmas \ref{case1} and \ref{case2}). In addition, we show that under certain conditions (that conditionally hold for infinitely many primes), the sequence $[P,Q]$ has a somewhat uniform distribution modulo $p$ in the sense that it contains every nonzero element the same number of times over the full period (see Theorem \ref{uniform1} for details). 

In Section \ref{completepairs}, we fix a prime and consider the complete generalized Fibonacci sequences with parameters $(P,Q) \in \{1,\ldots, p-1\}^2$. We show that the relative size of the number of such pairs asymptotically can approach but never exceed $\frac{1}{2}$.


\section{Preliminary Results}\label{prelemsection}

Let $P,Q$ be integers and $p>2$ be a prime. The solution $F=[P,Q]$ to the equation \eqref{gfequation} can be written by Binet's formula as
\begin{equation}\label{binet}
F_n=\frac{\alpha^n-\beta^n}{\alpha-\beta}, ~\forall n \geq 0,
\end{equation}
where 
\begin{equation}\label{roots}
\alpha = \frac{P + \sqrt{\Delta}}{2},~\beta = \frac{P - \sqrt{\Delta}}{2},
\end{equation}
are the roots of the characteristic polynomial $x^2-Px+Q$, and $\Delta=P^2-4Q \neq 0$ is its discriminant. One has
\begin{equation}\label{quadformula}
\alpha+\beta=P,~\alpha \beta=Q.
\end{equation}

Given a prime $p$, we will use $\hat{F}\in \{0,\ldots, p-1\}$ to denote the residue 
class of an integer $F$ modulo $p$. The sequence $\{\hat{F_n}\}_{n=0}^\infty$ modulo $p$ is a periodic sequence in $\mathbb{Z}_p$. A period of $[P,Q]$ modulo $p$ is an integer $n$ such that $F_{n+k} \equiv F_k \pmod p$ for all $k\geq 0$. The \emph{Pisano} period $\pi(p)$ is the least period of $[P,Q]$ modulo $p$, and it divides every period of $[P,Q]$.  We call the finite sequence $\{\hat{F_i}: 1\leq i \leq \pi(p)\}$ the \emph{full period} of $[P,Q]$ modulo $p$.
\begin{example}\label{147}
Let $P=1$, $Q=4$, and $p=7$. Then $\pi(p)=24$ and 
$$
\{\hat{F_i}: 1\leq i \leq 24\}=\{1, 1, 4, 0, 5, 5, 6, 0, 4, 4, 2, 0, 6, 6, 3, 0, 2, 2, 1, 0, 3, 3, 5, 0\}
$$
is the full period of $[P,Q]$ modulo $p=7$. 
\end{example}

The difference equation (\ref{gfequation}) has the following generalization which follows from equations \eqref{binet} and \eqref{quadformula} by a straightforward calculation. 
\begin{equation}\label{generalized}
F_n = F_aF_{n+1-a} - QF_{a-1}F_{n-a},~\forall n\geq 1, \forall a\in\{1,\ldots,n\}.
\end{equation}

Regarding the discriminant $\Delta$, we have three cases.
\begin{itemize}
\item[I.] $\Delta$ is a nonzero quadratic residue modulo $p$. In this case, the characteristic polynomial $x^2-Px+Q$ has two distinct roots modulo $p$ that we again denote by $\alpha$ and $\beta$. Binet's formula \ref{binet} still holds, where all the operations involved are carried out in the field $\mathbb{Z}_p$. By Fermat's little theorem, we have
$$F_{n+p-1}= \frac{\alpha^{n+p-1}-\beta^{n+p-1}}{\alpha-\beta}=\frac{\alpha^n-\beta^n}{\alpha-\beta}=F_n,$$
in $\mathbb{Z}_p$. It follows that $\pi(p) \mid (p-1)$, hence the sequence $[P,Q]$ is not complete modulo $p$ in this case.

\item[II.] $\Delta \equiv 0 \pmod p$. In this case, the characteristic polynomial $x^2-Px+Q$ has the repeated root $\alpha=\beta=P/2$. A simple induction shows
\begin{equation}\label{repeatedroot}
F_n \equiv n\left (\frac{P}{2} \right )^{n-1},~\forall n\geq 1.
\end{equation}
If $P \equiv 0 \pmod p$ (hence $Q \equiv 0 \pmod p$), then clearly $[P,Q]$ is not complete modulo $p$. If $P \not\equiv 0 \pmod p$, by letting $n=(p-1)k+1$ for an arbitrary integer $k$ and using Fermat's little theorem again, we conclude that $F_n \equiv -k+1 \pmod p$, hence $[P,Q]$ is complete modulo every prime $p$ in this case. 

\item[III.] $\Delta$ is a quadratic nonresidue modulo $p$. We consider the field extension
$$\mathbb{K}_p=\{a+b \sqrt{\Delta}: a, b\in \mathbb{Z}_p\}.$$
Binet's formula \eqref{binet} still holds with all the operations carried out in the field $\mathbb{K}_p$. Given an element $ \gamma \in \mathbb{K}_p$, the order of $\gamma$, denoted by $ord_p(\gamma)$ is the least positive integer $t$ such that $\gamma^t=1$. In particular, if $\gamma \in \mathbb{Z}_p$, then $ord_p(\gamma) \mid (p-1)$. Generally $ord_p(\gamma) \mid (p^2-1)$, since the multiplicative group $\mathbb{K}_p^*$ has order $p^2-1$. In particular, for $\alpha, \beta$ defined by equations \eqref{roots}, one has $\alpha^{p^2-1}=\beta^{p^2-1}=1$ in $\mathbb{K}_p$, and so $F_{n+p^2-1} \equiv F_n \pmod p$ by Binet's formula \eqref{binet}. Therefore, $\pi(p) \mid (p^2-1)$, which can be improved as shown in the following lemma.
\end{itemize}

\begin{lemma}\label{prelem}
\label{orders}
Suppose that $P,Q$ are integers and $p$ is an odd prime such that $\Delta=P^2-4Q$ is a quadratic nonresidue modulo $p$. Let $\alpha, \beta$ be the roots of $x^2-Px+Q=0$ in $\mathbb{K}_p$ given by \eqref{roots}. Then
\begin{itemize}
\item[(i)] $\alpha^{p+1}=\beta^{p+1}=Q$ in $\mathbb{K}_p$.

\item[(ii)] $\pi(p) \mid {(p+1)ord_p(Q)}$. Moreover, if $ord_p(Q)$ is even, then $\frac{(p+1) ord_p(Q)}{\pi(p)}$ is an odd integer.
\end{itemize}
\end{lemma}
\begin{proof}
(i) Let $\alpha= u + v\sqrt{\Delta}$, where $u=P/2,v=1/2$. We note that $ (\sqrt{\Delta})^{p+1} ={\Delta}^{(p-1)/2} {\Delta}= -\Delta$ by Euler's criterion. It follows from the binomial theorem modulo $p$ and Fermat's little theorem that
\begin{eqnarray} \nonumber
\alpha^{p+1} &=& u^{p+1} + u^pv\sqrt{\Delta} + uv^p(\sqrt{\Delta})^p + v^{p+1}(\sqrt{\Delta})^{p+1}  \\ \nonumber
&=&  u^2 + uv\sqrt{\Delta} - uv\sqrt{\Delta}- v^2\Delta \\ \nonumber
&=&u^2 - v^2\Delta = \frac{P^2}{4} - \frac{P^2 - 4Q}{4} = Q,
\end{eqnarray}
where all of the operations are carried out in the field $\mathbb{K}_p$. Similarly, $\beta^{p+1}=Q$ in $\mathbb{K}_p$. 

\noindent (ii) Let $t=ord_p(Q)$ such that $Q^{t} \equiv 1 \pmod p$. It follows from part (i) that $\alpha^{(p+1)t}=\beta^{(p+1)t}=Q^t=1$, and so $F_{(p+1)t+k} \equiv F_{k} \pmod p$ for all $k\geq 0$, which implies that $\pi(p) \mid (p+1)t$. 

 If $t$ is even, then $\alpha^{(p+1)t/2} =\beta^{(p+1)t/2}=Q^{t/2} =-1$. It follows again from Binet's formula \eqref{binet} that $F_{(p+1)t/2+k}=-F_{k}$ for all $k\geq 0$. In particular, $\pi(p) \nmid \frac{(p+1)t}{2}$, and therefore, $\frac{(p+1)t}{\pi(p)}$ is odd. 
\end{proof}

In the rest of this section, we briefly discuss the cases that are not covered by Theorem \ref{main} i.e., when $PQ =0$ or $P^2-4Q$ is a perfect square.
\begin{itemize}
\item[i.] $P=0$. Then $F_{2k}=0$ and $F_{2k+1}=(-Q)^{k}$ for all $k\geq 0$. Therefore, $[0,Q]$ is complete modulo $p>2$ if and only if $-Q$ is a primitive root modulo $p$. 
\item[ii.] $Q=0$. Then $F_{n}=P^{n}$ for all $n\geq 1$. Therefore, $[P,0]$ is complete modulo $p>2$ if and only if $P$ is a primitive root modulo $p$. 
\item[iii.] $P^2-4Q=0$. If $P \equiv 0 \pmod p$, the sequence is clearly not complete. If $P \not\equiv 0 \pmod p$, then the sequence is given by \eqref{repeatedroot} which is complete modulo $p$ as shown earlier in Case II. 
\item[iv.] $P^2-4Q$ is a nonzero perfect square. Then by Binet's formula \eqref{binet} and Fermat's little theorem, the sequence $[P,Q]$ is not complete modulo $p$ for all primes $p>2$ as shown earlier in Case I. 
\end{itemize}
Therefore, we have the following corollary of Theorem \ref{main}.
\begin{cor}
(GRH) Let $P,Q$ be arbitrary integers. Then the sequence $[P,Q]$ is complete modulo infinitely many primes if and only if one of the following statements is true.
\begin{itemize}
\item[i.] $P\neq 0$ and $P^2-4Q=0$. 
\item[ii.] $Q=0$ and $P \neq \pm 1$ and $P$ is not a perfect square.
\item[iii.] $Q \neq \pm 1$ and $P^2-4Q$ is not a perfect square. 
\end{itemize}
\end{cor}


\section{Proof of the main theorem}\label{main}

First, we consider generalized Fibonacci sequences $[P,1]$. 

\begin{theorem}\label{caseq1}
Let $P$ be an integer and $p>3$ be a prime. Then the sequence $[P,1]$ is complete modulo $p$ if and only if $P \equiv \pm 2 \pmod p$. \end{theorem}
\begin{proof}
If $P^2-4$ is zero or a nonzero quadratic residue modulo $p$, the claim follows from our analysis of the three cases (I-III) in Section \ref{prelemsection}. Thus, suppose that $P^2-4$ is a quadratic nonresidue modulo $p$ and $F=[P,1]$ is complete modulo $p$, and we derive a contradiction. It follows from part (i) of Lemma \ref{prelem} that, in this case $\alpha^{p+1}=\beta^{p+1}=1$. It follows that, for $r=(p+1)/2$, we have $\alpha^{r}=\pm 1$ and $\beta^r=\pm 1$. Since $\alpha \beta =Q=1$, one has $\alpha^r \beta^r=(\alpha \beta)^r=1$, hence $\alpha^r=\beta^r=\pm 1$. If $\alpha^r=\beta^r=1$, then $F_{r+n}\equiv F_n \pmod p$ for all $n\geq 0$, hence $\pi(p) \mid r$ and so the sequence is not complete modulo $p$. Therefore, suppose $\alpha^r=\beta^r=-1$. Then, for all $n\geq 0$, 
$$F_{r+n} =\frac{\alpha^{r+n}-\beta^{r+n}}{\alpha-\beta}=\frac{-\alpha^n+\beta^n}{\alpha-\beta}= -F_n,$$ 
in $\mathbb{Z}_p$. In particular, by our assumption that $[P,1]$ is complete, we must have $\{\pm F_1,\ldots, \pm F_{r-1}\}=\{1,\ldots, p-1\}$ modulo $p$. It follows that 
\begin{equation}\label{sumsquares}
F_1^2+\ldots+F_{r-1}^2 \equiv 1^2+\ldots+(r-1)^2 \equiv 0 \pmod p,
\end{equation} for $p>3$. On the other hand,
\begin{eqnarray}\nonumber
\sum_{i=1}^{r-1}F_i^2 &=& \frac{1}{(\alpha-\beta)^2}\sum_{i=1}^{r-1}(\alpha^i-\beta^i)^2 \\ \nonumber
&=&\frac{1}{(\alpha-\beta)^2} \left ( \frac{\alpha^{2r}-\alpha^2}{\alpha^2-1}-2\sum_{i=1}^{r-1}Q^i+\frac{\beta^{2r}-\beta^2}{\beta^2-1} \right ) \\ \nonumber
&=& \frac{1}{\Delta} \left ( -2-2\sum_{i=1}^{r-1} 1  \right )=\frac{-1}{\Delta},
\end{eqnarray}
which contradicts \eqref{sumsquares}.
\end{proof}

Part (i) of Theorem \ref{main} follows from Schinzel's result \cite{Schinzel} for $Q=-1$ and from Theorem \ref{caseq1} for $Q=1$. Note that Theorem \ref{caseq1} implies that if $\Delta=P^2-4 \neq 0$ then $[P,1]$ is not complete modulo any $p>|P|+3$.

The \emph{rank of apparition} $\rho=\rho(p)$ is the least positive integer $n$ such that $F_n \equiv 0 \pmod p$, or equivalently $\alpha^n=\beta^n$ in $\mathbb{K}_p$. If $\alpha^m=\beta^m$ in $\mathbb{K}_p$ for an integer $m$, then $\rho \mid m$. One can arrange the full period $\{\hat{F}_i: 1\leq i \leq \pi(p)\}$ in a $\frac{\pi}{\rho}\times \rho$ matrix, where the $ij$th entry is given by $\hat{F}_{(i-1)\rho+j}$, $1\leq i \leq \pi/\rho$, $1\leq j \leq \rho$; We call this matrix the {\it period matrix} of $[P,Q]$ modulo $p$. Since $F_{t} \equiv 0 \pmod p \Leftrightarrow \rho \mid t$, the last column of the period matrix is a zero column and all other entries are nonzero modulo $p$. It follows that 0 appears exactly $\pi/\rho$ times in the full period. In the next theorem, we show that if $Q$ is a primitive root or the square of a primitive root modulo $p$, then every nonzero element appears the same number of times in the full period. 

\begin{theorem}\label{uniform1}
Let $p$ be an odd prime such that $p \nmid P$ 
and $P^2-4Q$ is a quadratic nonresidue modulo $p$. Suppose that $ord_p(Q)=\frac{p-1}{k}$ for 
some $k\in \{1,2\}$. Then the generalized Fibonacci sequence $[P,Q]$ is complete modulo $p$. More precisely, the following statements hold.
\begin{itemize}
\item[]{a)} If $p \equiv 1 \pmod 4$, or $p \equiv 3 \pmod 4$ and $\rho(p)$ is even, 
then zero appears in the full period $\{\hat{F}_i: 1\leq i \leq \pi(p)\}$ exactly $p-1$ times, 
while each nonzero element of $\mathbb{Z}_p$ appears exactly $\rho -1 $ times. 

\item[]{b)}  If $p \equiv 3 \pmod 4$ and the rank of apparition $\rho$ is odd, then 
\subitem{}i) either zero appears in the full period exactly 
$p-1$ times, while each nonzero element of $\mathbb{Z}_p$ appears $\rho -1 $ times, or
\subitem{}ii) zero appears in the full period exactly $(p-1)/2$ times, while each nonzero element of $\mathbb{Z}_p$ appears 
exactly $(\rho -1)/2$ times. 
\end{itemize}
\end{theorem}
\begin{proof}
a) Let $u=\alpha^\rho=F_{\rho+1} \in \mathbb{Z}_p$. We first show that $u$ is a primitive root modulo $p$. 
Since $F_{p+1}=\alpha^{p+1}-\beta^{p+1}=0$ by Lemma \ref{prelem}, we have $\rho \mid (p+1)$. Let $l=(p+1)/\rho$, and so $u^l=Q$ in $\mathbb{K}_p$. If $k=1$, then $Q$ is a primitive root modulo $p$, hence $u$ is a primitive root modulo $p$. Thus, suppose that $k=2$, hence $Q$ is a quadratic residue modulo $p$. It follows from $u^l=Q$ that $ord_p (u) =p-1$ or $\frac{p-1}{2}$. Since $(\alpha^{(p+1)/2})^2=Q$, and $Q$ is a quadratic residue, we must have $\alpha^{(p+1)/2} \in \mathbb{Z}_p \subseteq \mathbb{K}_p$. Similarly $\beta^{(p+1)/2} \in \mathbb{Z}_p$. Therefore, either $\alpha^{(p+1)/2}=\beta^{(p+1)/2}$ or $\alpha^{(p+1)/2}=-\beta^{(p+1)/2}$. If $\alpha^{(p+1)/2}=-\beta^{(p+1)/2}$, we have $F_{(p+3)/2}=\alpha^{(p+1)/2}P/(\alpha-\beta) \in \mathbb{Z}_p$, which is a contradiction, since $P \not\equiv 0$ modulo $p$ and $\alpha-\beta=\sqrt{\Delta} \notin \mathbb{Z}_p$. 
Therefore, we have $\alpha^{(p+1)/2}= \beta^{(p+1)/2}$, and so $\rho \mid \frac{p+1}{2}$. There are two cases.
\\
\\
Case 1. If $p \equiv 1 \pmod 4$, then it follows from $\rho \mid \frac{p+1}{2}$ that $\rho$ is odd and $l$ is even. But then, $(u^{(p-1)/2})^{l/2}=Q^{(p-1)/4}=-1$, which implies that $ord_p (u) \neq \frac{p-1}{2}$, and so $ord_p(u)=p-1$ in this case. 
\\
\\
Case 2. If $p \equiv 3 \pmod 4$, then $\rho$ is even by the assumption. If $ord_p(u)=\frac{p-1}{2}$, then $u$ is a quadratic residue and so $\alpha^{\rho/2}, \beta^{\rho/2} \in \mathbb{Z}_p$, since $(\alpha^{\rho/2})^2=(\beta^{\rho/2})^2=u$. It follows that $\alpha^{\rho/2}=\pm \beta^{\rho/2}$. However, by definition of $\rho$, we cannot have $\alpha^{\rho/2}=\beta^{\rho/2}$. Therefore, $\alpha^{\rho/2}=-\beta^{\rho/2}$, which implies that $F_{\rho/2+1}=\alpha^{\rho/2}P/(\alpha-\beta)$ which contradicts $\alpha-\beta=\sqrt{\Delta} \notin \mathbb{Z}_p$. It follows that $ord_p(u)=p-1$ in this case as well.

We have so far shown that $u=\alpha^\rho$ is a primitive root modulo $p$. In particular, the numbers $F_{s\rho+1}=\alpha^{s\rho}$, $1\leq s \leq p-1$ form a reduced residue system modulo $p$ i.e., they contain every nonzero congruence class modulo $p$. 
Since for each $1\leq s \leq p-1$ and $0\leq t <\rho$, we have
$$F_{s\rho+t}=\frac{\alpha^{s\rho+t}-\beta^{s\rho+t}}{\alpha-\beta}=\alpha^{s\rho}F_t=F_{s\rho+1}F_t,$$
and $F_t \neq 0$ for $1\leq t <\rho$, we conclude that the numbers $F_{s\rho+t}$, $1\leq s \leq p-1$ 
form a reduced residue system modulo $p$ for each fixed $0 < t <\rho$. It follows that zero 
appears exactly $p-1$ times among the numbers $\hat{F}_1,\ldots, \hat{F}_{\pi(p)}$, 
and each $k\in \{1,\ldots, p-1\}$ 
appears exactly $\rho-1$ times in the full period. Moreover, $\pi(p)=\rho(p-1)$ i.e., 
each nonzero element of $\mathbb{Z}_p$ appears exactly $\rho-1=\frac{\pi(p)-p+1}{p-1}$ times. This completes the proof of part (a).
\\
\\
\emph{Proof of part (b).} If $ord_p(u)=p-1$, then again $u$ is a primitive root modulo $p$ and the claim follows as in the proof of part (a). Thus, suppose that  
$ord_p(Q)=ord_p(u)=\frac{p-1}{2}$. Then $\{F_{s\rho+1}=u^{s}: 1\leq s \leq( p-1)/2\}$ contains exactly all quadratic residues modulo $p$. 
Since $F_{a+ s\rho} = u^sF_a$, we have $\legendre{F_a}{p} = \legendre{F_{a+s\rho}}{p}$. Therefore, the last column of the period matrix is a zero column; moreover, every other column of the period matrix lists either exactly the set of all nonzero quadratic residues or the set of all quadratic nonresidues, depending on whether the first entry in the column is a quadratic residue or a quadratic nonresidue. Equation \eqref{generalized} implies that $F_a F_{\rho-(a-1)} \equiv Q F_{a-1}F_{\rho-a} \pmod p$. Since $Q$ is a quadratic residue modulo $p$, it follows that
$$
\legendre{F_a}{p}\legendre{F_{\rho-a}}{p}=\legendre{F_{a-1}}{p}\legendre{F_{\rho-(a-1)}}{p}, ~ \forall 
a\in\{2,\ldots,\rho-1\}.
$$
Therefore, $\legendre{F_a}{p}\legendre{F_{\rho-a}}{p}$ is independent of $a$, and so
$$\legendre{F_a}{p}\legendre{F_{\rho-a}}{p}=\legendre{F_{\rho-1}}{p}\legendre{F_1}{p}=-1,$$
since $u =\alpha^\rho=F_{\rho+1}=PF_\rho-QF_{\rho-1}=-QF_{\rho-1}$ which implies that $\legendre{F_{\rho-1}}{p}=-1$. 
We conclude that for each $a\in \{1,\ldots, \rho-1\}$ exactly one of the numbers $F_a$ or $F_{\rho-a}$ is a nonzero quadratic residue modulo $p$. So half of the $(\rho-1)$ nonzero columns of the period matrix list exactly the set of all nonzero quadratic residues and the other half list exactly the set of quadratic nonresidues. It follows that each $k\in\{1,2,\ldots, p-1\}$ appears $(\rho - 1)/2$ times in the full period.
\end{proof}

To illustrate the statements of Theorem \ref{uniform1}, first let $(P,Q)=(1,4)$ 
and $p=13$. In this case $ord_{13}(4)=6$, $\pi(13)=84$, and $\rho(13)=7$. The 
period matrix is given by  
$$
\begin{pmatrix}
1 & 1 & 10 & 6 & 5 & 7 & 0\\ 
11 & 11 & 6 & 1 & 3 & 12 & 0\\ 
4 & 4 & 1 & 11 & 7 & 2 & 0\\
5 & 5 & 11 & 4 & 12 & 9 & 0\\
3 & 3 & 4 & 5 & 2 & 8 & 0\\ 
7 & 7 & 5 & 3 & 9 & 10 & 0\\ 
12 & 12 & 3 & 7 & 8 & 6 & 0\\ 
2 & 2 & 7 & 12 & 10 & 1 & 0\\ 
9 & 9 & 12 & 2 & 6 & 11 & 0\\ 
8 & 8 & 2 & 9 & 1 & 4 & 0\\ 
10 & 10 & 9 & 8 & 11 & 5 & 0\\ 
6 & 6 & 8 & 10 & 4 & 3 & 0\\
\end{pmatrix}=\begin{pmatrix}
1 \\ 
11 \\ 
4 \\
5 \\
3 \\ 
7 \\ 
12\\ 
2 \\ 
9 \\ 
8 \\ 
10 \\ 
6 \\
\end{pmatrix} \begin{pmatrix}
1 & 1 & 10 & 6 & 5 & 7 & 0
\end{pmatrix} \pmod{13}
$$
For part (b) of the Theorem \ref{uniform1}, let $(P,Q) =  (12,5)$ and $p = 19$. 
Here $\rho(19) = 5$, $ord_{19}(5) = 9$, $\pi(19) = 45$, and 
$u = \alpha^5 = 103566 + 18601\sqrt{31} \equiv 16\pmod{19}$. 
The period matrix is given by
$$
\begin{pmatrix}
1 & 12 & 6 & 12 & 0\\ 
16 & 2 & 1 & 2 & 0\\
9 & 13 & 16 & 13 & 0\\
11 & 18 & 9 &18 & 0\\
5 & 3 & 11 & 3 & 0\\
4 & 10 & 5 & 10 & 0\\
7 & 8 & 4 & 8 & 0\\
17 & 14 & 7 & 14 & 0\\
6 & 15 & 17 & 15 & 0\\
\end{pmatrix} =\begin{pmatrix}
1 \\ 
16 \\ 
9 \\
11 \\
5 \\ 
4 \\ 
7 \\ 
17 \\ 
6 \\ 
\end{pmatrix} \begin{pmatrix}
1 & 12 & 6 &12 & 0 
\end{pmatrix} \pmod{19}
$$
The first and the third columns of this matrix list exactly all of nonzero quadratic residues modulo 19, while the 
second and the fourth columns list exactly all of the quadratic nonresidues modulo 19.

As stated in the introduction, if $Q \neq \pm 1$ and $Q$ is not a perfect square, and if Artin's conjecture holds, then by Somer's result \cite{Somer} the sequence $[P,Q]$ is complete modulo infinitely many primes. To complete the proof of Theorem \ref{main}, we next consider the case where $Q$ is a perfect square not equal to 1. 

\begin{lemma}\label{case1}
Suppose that $Q=m^2$ for an integer $m>1$. Suppose that $\pm(P^2-4Q)$ are not perfect squares. Then $[P,Q]$ is complete modulo infinitely many primes. 
\end{lemma}

\begin{proof}
Let $p_1,\ldots, p_l$ be primes such that $\Delta=P^2-4Q=\pm X^2p_1\ldots p_l$, where $X$ is an integer and $p_1 <p_i$ for all $i>1$. We choose a number $t_1$ such that $\legendre{t_1}{p_1}=-1$ if $p_1$ is odd, and let $t_1=5$ if $p_1=2$. By the Chinese remainder theorem, there exists an integer $T$ such that $T \equiv 5 \pmod 8$ and $T \equiv t_1 \pmod {p_1}$ and $T\equiv 1 \pmod{p_i}$ for all $i>1$. For every prime $p$ in the arithmetic progression ${\mathcal A}=\{8kp_1\cdots p_l+T: k\geq 0\}$, we have
$$\legendre{\Delta}{p}=\legendre{p_1}{p}=\begin{cases} \legendre{p}{p_1}=\legendre{t_1}{p_1}=-1 & \mbox{if $p_1$ is odd}; \\
\legendre{2}{p}=-1& \mbox{if $p_1=2$}.
\end{cases}$$
By \cite[Theorem 4]{Moree}, there exist infinitely many primes $p$ in the arithmetic progression $\mathcal A$ for which $ord_p(-m)=p-1$. It follows that for such $p$ we have that $ord_p(Q)=\frac{p-1}{2}$ and $\Delta$ is a quadratic nonresidue modulo $p$. By Theorem \ref{uniform1}, the sequence $[P,Q]$ is complete modulo each such $p>P$, and the claim follows.
\end{proof}

 \begin{lemma}\label{case2}
Suppose that $Q=m^2$ with $m>1$ and $P^2-4Q=-X^2$, where $X$ is a nonzero integer. Then $[P,Q]$ is complete modulo infinitely many primes. 
 \end{lemma}
 \begin{proof}
Let $\mathcal A=\{4k+3: k\geq 0\}$. By \cite[Theorem 4]{Moree}, there exist infinitely many primes $p$ in the arithmetic progression $\mathcal A$ for which $ord_p(-m)=p-1$. It follows that $ord_p(Q)=\frac{p-1}{2}$ for such $p$ and $\Delta=-X^2$ is a quadratic nonresidue modulo $p$, since $p \equiv 3 \pmod 4$. By Theorem \ref{uniform1}, the sequence $[P,Q]$ is complete modulo each such $p>P$, and the claim follows.
  \end{proof}
  
 Part (ii) of Theorem \ref{main} follows from Lemmas \ref{case1} and \ref{case2}. Moree's result \cite[Theorem 4]{Moree} on primitive roots modulo primes in an arithmetic progression is a conditional result, and so our proof of part (ii) of Theorem \ref{main} is conditional to GRH. 
 

 \section{The relative size of complete pairs}\label{completepairs}
 
 In Theorem \ref{main}, we fix the parameters $P,Q$ and consider the primes $p$ modulo which the generalized Fibonacci sequence $[P,Q]$ is complete. In this section, we fix a prime $p$ and consider parameters $P,Q$ for which the sequence $[P,Q]$ is complete. Let $\Lambda_p$ denote the set of pairs $(P,Q) \in \{1,\ldots, p-1\}^2$ such that $[P,Q]$ is complete modulo $p$. We are particularly interested in the relative size of $\Lambda_p$ in the set $\{1,\ldots, p-1\}^2$.  
 
 In what follows, let ${\mathcal A}_p$ denote the set of primitive roots modulo $p$ in $\mathbb{Z}_p$. 
 Let ${\mathcal B}_p$ denote the set of $Q \in \{1,\ldots, p-1\}$ such that $ord_p(Q)=\frac{p-1}{k}$, where $k\in \{1,2\}$. In particular ${\mathcal A}_p \subseteq {\mathcal B}_p$. Finally, for each $Q \in \{1,\ldots, p-1\}$, let ${\mathcal C}_Q$ denote the set of $P \in \{1,\ldots, p-1\}$ such that $P^2-4Q$ is a quadratic nonresidue.

 \begin{lemma}\label{boundbp}
For every odd prime $p$, one has ${|\Lambda_p|} \geq \frac{1}{2} (p-3)|{\mathcal B}_p |.$  
 \end{lemma}
 
 \begin{proof}
  Let $\mathcal X$ be the set of $x\in \{1,\ldots, p-1\}$ such that $x$ is a quadratic residue and $x-1$ is a quadratic nonresidue, and let $\mathcal Y$ be the set of $y\in \{1,\ldots, p-1\}$ such that $y-1$ is a quadratic residue and $y$ is a quadratic nonresidue modulo $p$. It follows from a result of Aladov \cite[Theorem 1]{Aladov} that 
  \begin{equation}\label{xyineq}
|\mathcal X|, |\mathcal Y| \geq \frac{p-3}{4}.
\end{equation}
If $Q$ is a quadratic residue modulo $p$, then the map $P \mapsto P^2/(4Q)$ is a two-to-one map from ${\mathcal C}_Q$ onto $\mathcal X$. If $Q$ is a quadratic nonresidue modulo $p$, then the map $P \mapsto P^2/(4Q)$ is a two-to-one map from ${\mathcal C}_Q $ onto $\mathcal Y$. In either case, it follows from Inequalities \eqref{xyineq} that
$$|{\mathcal C}_Q|\geq \frac{p-3}{2}.$$
By Theorem \ref{uniform1}, if $Q\in {\mathcal B}_p$ and $P \in {\mathcal C}_Q$, then $[P,Q]$ is complete modulo $p$. Therefore, we have
$${|\Lambda_p|} \geq \sum_{Q \in {\mathcal B}_p} |{\mathcal C}_Q| \geq \frac{1}{2}(p-3)|{\mathcal B}_p |.$$
This completes the proof of Lemma \ref{boundbp}.
 \end{proof}

Now, we are ready to prove the main theorem of this section. 
 
 \begin{theorem}\label{limsupap}
 $$\limsup_{p \rightarrow \infty} \frac{|\Lambda_p|}{p^2}=\frac{1}{2}.$$
 \end{theorem}
 
 \begin{proof}
 By \cite[Lemma 1]{HB}, there exist infinitely many primes $p_i \equiv 3 \pmod 4$, $i\geq 1$, such that $\frac{p_i-1}{2}$ is either a prime or a product of two primes $q_i,r_i$ such that $q_i,r_i>p_i^{1/4}$. Therefore
 $$\frac{1}{2} > \dfrac{\phi(p_i-1)}{p_i-1}>\frac{1}{2}\left (1-\frac{1}{p_i^{1/4}} \right )^2,$$
for all $i\geq 1$. Since the number of primitive roots modulo $p_i$ is given by $|{\mathcal A}_{p_i}|=\phi(\phi(p_i))=\phi(p_i-1)$, it follows that
 \begin{equation}\label{limsupa}
\lim_{i \rightarrow \infty}\frac{|{\mathcal A}_{p_i}|}{p_i}=\lim_{i \rightarrow \infty}\frac{\phi(p_i-1)}{p_i-1}=\frac{1}{2}.
\end{equation}
For every prime $p \equiv 3 \pmod 4$, the map $x \mapsto x^2$ is an injective map from ${\mathcal A}_p$ into ${\mathcal B}_p \backslash {\mathcal A}_p$. It follows that $|{\mathcal B}_p|=|{\mathcal A}_p|+|{\mathcal B}_p \backslash {\mathcal A}_p|\geq 2 |{\mathcal A}_p|$, and so 
 $$\lim_{i \rightarrow \infty}\frac{|{\mathcal B}_{p_i}|}{p_i} = 1.$$
Therefore, by Lemma \ref{boundbp}, we conclude that
\begin{equation}\label{limsupl}
\limsup_{p \rightarrow \infty}\frac{|\Lambda_{p}|}{p^2}  \geq \limsup_{i \rightarrow \infty}\frac{|\Lambda_{p_i}|}{p_i^2} \geq \lim_{i \rightarrow \infty}\frac{\frac{1}{2}(p_i-3)|{\mathcal B}_{p_i}|}{p_i^2} \geq  \frac{1}{2}.
\end{equation}
On the other hand, for every odd prime $p$, one has $|\Lambda_p| \leq \frac{1}{2}(p^2-1)$. To see this, we note that for each $P \in \{1,\ldots, p-1\}$, the number of $Q$ such that $P^2-4Q$ is zero or a quadratic nonresidue is $(p+1)/2$. It follows that for at most $\frac{1}{2}(p^2-1)$ pairs $[P,Q] \in \{1,\ldots, p-1\}^2$, the sequence $[P,Q]$ is complete modulo $p$, and so $|\Lambda_p| \leq \frac{1}{2}(p^2-1)$, which implies that $\limsup_{p \rightarrow \infty} |\Lambda_p|/p^2 \leq \frac{1}{2}$. This together with \eqref{limsupl} completes the proof of Theorem \ref{limsupap}.
  \end{proof}

\end{document}